\documentclass[11pt]{article}

\usepackage[T1]{fontenc}
\usepackage[margin=1in]{geometry}
\usepackage{amsmath,amssymb,amsfonts,amsthm,mathtools,bm}
\usepackage{mathpazo}
\usepackage[scaled]{helvet}
\usepackage{courier}
\normalfont
\usepackage{microtype}
\usepackage[dvipsnames]{xcolor}
\usepackage{graphicx,caption}
\usepackage[authoryear,round]{natbib}
\usepackage[shortlabels]{enumitem}
\usepackage{tikz,pgfplots}
\usetikzlibrary{arrows.meta}
\pgfplotsset{compat=1.18}
\usepackage{url}
\usepackage{hyperref}
\hypersetup{
  colorlinks=true,
  citecolor=MidnightBlue,
  linkcolor=red!80!black,
  urlcolor=MidnightBlue,
  pdftitle={Nonminimizing Attraction of Stochastic Subgradient Methods on Semialgebraic Functions},
  pdfauthor={Jiajin Li}
}
\allowdisplaybreaks[4]
\theoremstyle{plain}
\newtheorem{theorem}{Theorem}[section]
\newtheorem{proposition}[theorem]{Proposition}

\newtheorem{corollary}[theorem]{Corollary}
\theoremstyle{definition}

\newtheorem{remark}[theorem]{Remark}
\numberwithin{equation}{section}

\newcommand{\R}{\mathbb R}

\title{Nonminimizing Attraction of Stochastic Subgradient Methods\\
on Semialgebraic Functions}
\author{%
Jiajin Li\\[0.75em]
Sauder School of Business,\\
University of British Columbia\\[0.5em]
\small\texttt{jiajin.li@sauder.ubc.ca}
}
\date{\today}

\begin{document}

\maketitle

\begin{abstract}
We construct a globally Lipschitz, coercive, semialgebraic function in two dimensions for which stochastic subgradient iterates converge to the origin, a Clarke critical point that is neither limiting critical nor a local minimizer. For power-law stepsizes, this attraction holds on open sets of initial points and linear tilts, uniformly over all noise sequences within a prescribed bound. The example shows that subdifferential regularity cannot be removed from the generic guarantee of convergence to local minimizers established by \citet*{DDJ2023}. For the same tilted objectives and initial points, the unique Clarke subgradient flow converges to a strict local minimizer. Thus, convergence of continuous-time flows to local minimizers cannot by itself justify the corresponding guarantee for discrete subgradient methods, even with diminishing stepsizes.
\end{abstract}



\section{Introduction}\label{sec:introduction}
For generic smooth optimization problems, stochastic gradient
iterates that remain bounded converge almost surely to a local
minimizer, provided the stepsizes decay suitably and the noise
excites every direction.
Here genericity refers to almost every linear tilt of the
objective, for which all critical points are nondegenerate.
Every nonminimizing critical point then has a negative Hessian
eigenvalue, creating an unstable direction for the gradient
dynamics.
The ODE method of \citet{Benaim1999} shows that
bounded iterates converge to a critical point, and the classical
theory of \citet{Pemantle1990} rules out convergence to
such unstable equilibria.
Later, \citet{LPPSJR2019} used stable manifold theory
to show that deterministic gradient descent on $C^2$ objectives
with Lipschitz gradients avoids strict saddles from almost every
initialization with a sufficiently small fixed stepsize.

The stochastic subgradient method extends stochastic gradient
descent to nonsmooth objectives, where Clarke criticality
provides a natural notion of stationarity \citep{Clarke1990}.
For general Lipschitz functions, even convergence to Clarke
critical points can fail: \citet{DD2020}
construct bounded deterministic subgradient iterates whose
accumulation points are all noncritical, despite nonsummable,
square-summable stepsizes.
Moreover, \citet{RZ2020} shows that bounded
subgradient iterates can fail to converge even for Whitney
stratifiable functions satisfying the Kurdyka--{\L}ojasiewicz
inequality.
For locally Lipschitz functions definable in an o-minimal
structure, including all semialgebraic functions,
\citet{DDKL2020} establish a positive result:
under their stepsize and noise assumptions, every accumulation
point of bounded stochastic subgradient iterates is almost
surely Clarke critical.
This conclusion concerns only stationarity. It leaves open
whether, as in the smooth case, noise prevents convergence
to nonminimizing critical points. In particular, does the
generic convergence guarantee from smooth optimization extend
to this nonsmooth setting?

For stochastic subgradient methods, \citet*{BHS2024} establish avoidance of active strict
saddles and a generic guarantee of convergence of bounded
trajectories to local minimizers for definable weakly convex
functions.
\citet*{DDJ2023} establish this
guarantee for the broader class of definable, locally Lipschitz,
subdifferentially regular functions.
Specifically, for almost every linear tilt
$f_{\bm{v}}(\bm{x}):=f(\bm{x})-\langle \bm{v},\bm{x}\rangle$,
on the event that the stochastic subgradient iterates remain
bounded, they converge almost surely to a local minimizer under
their stepsize and noise assumptions.
Their analysis relates the nonsmooth iteration to an
approximate Riemannian gradient method on an active manifold,
allowing them to apply classical arguments for avoiding
unstable critical points.
Without subdifferential regularity, their result for
semialgebraic functions~\citep[Corollary~5.6]{DDJ2023} leaves
two possibilities for almost every tilt: bounded iterates
converge almost surely either to a local minimizer or to a
Clarke critical point that need not be a local minimizer.
In the 2023 preprint version\footnote{\url{https://arxiv.org/abs/2108.11832v2}.},
immediately after Theorem~1.1, they conjectured that generic
convergence to local minimizers holds without subdifferential regularity.
\citet[Remark~4]{BHS2024} likewise note that
the generic convergence of stochastic subgradient methods to
local minimizers remains open beyond weak convexity.

We show that the second alternative does occur for an open set
of tilts, and hence resolve both in the negative.
Specifically, we construct a globally Lipschitz, coercive,
semialgebraic function in two dimensions for which the origin is Clarke
critical but neither limiting critical nor a local minimizer.
Although the origin admits directions of linear decrease, it
attracts stochastic subgradient iterates with stepsizes of order
$k^{-\gamma}$, for any $\gamma\in(\tfrac12,1]$, from an open set
of initial points.
This attraction holds for every tilt in this open set and
uniformly over all perturbation sequences within a prescribed bound.

This attraction comes from a cycle of the iteration after
rescaling by the stepsize.
Inspired by the BFGS example of \citet{LZ2015},
we build the instance from a positively homogeneous function
on $\mathbb{R}^2$ that is $C^\infty$ away from the origin.
Its gradient is constant along rays, so the rescaled iteration
is, up to small errors, the iteration of a fixed map with a
contracting four-cycle.
The rescaled iterates remain near this cycle, and the diminishing
stepsizes drive the original iterates to the origin.
Along these trajectories, the length of the gradient step remains
comparable to the distance to the origin.
The contracting cycle makes this mechanism robust to perturbations,
allowing attraction despite the directions of linear decrease.

This mechanism is genuinely discrete.
For the same initial points and tilts, the Clarke subgradient
flow avoids the origin and converges to a strict local minimizer;
see Theorem~\ref{thm:flow-separation}. 
The same holds from almost every initial point; see Remark \ref{rem:global-flow}.
The rescaled discrete
iterates remain near the four-cycle{, so they sample the objective only near the four cycle directions}, while the original iterates
converge to the nonminimizing origin.
{The flow, by contrast, must pass through the directions between them, where the objective decreases along rays.}
Classical avoidance results exploit linear instability under
smoothness and noise assumptions~\citep{Pemantle1990}; the
nonsmooth results of \citet{BHS2024} and \citet{DDJ2023} use regularity conditions to relate
the iterates to smooth dynamics on an active manifold.
Without regularity, a natural strategy is to show that the Clarke
subgradient flow converges to local minimizers and to transfer
this behavior to the iterates.
Our example shows that this transfer fails: even with diminishing
stepsizes, convergence of the flow to a local minimizer does not
imply the same for the discrete method.
{This is consistent with} the limit-set theory for differential
inclusions~\citep{BenaimHofbauerSorin2005}: since the
origin is Clarke critical, $\bm{x}(t)\equiv\bm{0}$ is a solution
of the flow.

\paragraph{Concurrent work.}
Independently and concurrently, \citet{LiHu2026}
obtained a counterexample to the same generic convergence guarantee.
Both constructions use a similar rescaling argument:
after rescaling by the stepsize, the iterates stay near a
periodic orbit, and the vanishing stepsizes carry them to a
nonminimizing point.
In our work, this mechanism comes from a general
{criterion}: for any positively homogeneous function, a
contracting cycle of the rescaled map forces convergence to
the origin under small bounded perturbations
(Proposition~\ref{prop:cycle-principle}).
We derive our example from this {criterion}, using a positively
homogeneous function that is smooth away from the origin.
Moreover, the two examples behave {differently} in
continuous time.
In their example, the flow reaches the nonminimizing origin
in finite time and admits a stationary continuation
(Remark~\ref{rem:li-hu-flow}).
In ours, every flow from the same initial points as the
discrete method converges to a strict local minimizer,
so the attraction to the origin is purely discrete.

\section{Preliminaries and Main Results}\label{sec:main-results}
\subsection{Preliminaries}\label{sec:preliminaries}
Vectors are written in bold. We use the Euclidean norm and its induced
operator norm.
The open and closed balls of radius $r$ centered at $\bm{x}$ are
$\mathbb B_r(\bm{x})$ and $\overline{\mathbb B}_r(\bm{x})$, respectively.
The symbols $\mathcal I$ and $D$ denote the identity map and the Jacobian.

For a locally Lipschitz function $h\colon\R^d\to\R$, define
\[
 \begin{aligned}
 \widehat\partial h(\bm{x})
 &:=\left\{\bm{w}:\liminf_{\substack{\bm{y}\to \bm{x}\\\bm{y}\ne \bm{x}}}
 \frac{h(\bm{y})-h(\bm{x})-\langle \bm{w},\bm{y}-\bm{x}\rangle}{\|\bm{y}-\bm{x}\|}\ge0\right\},\\
 \partial h(\bm{x})
 &:=\left\{\bm{w}:\exists\,\bm{x}_n\to \bm{x},\ \bm{w}_n\to \bm{w},
                    \ \bm{w}_n\in\widehat\partial h(\bm{x}_n)\right\},\\
 \partial_C h(\bm{x})&:=\operatorname{conv}\partial h(\bm{x}).
 \end{aligned}
\]
These are the Fr\'echet, limiting (Mordukhovich), and Clarke
subdifferentials; see \citet{Clarke1990}, \citet{Mordukhovich2006}, and \citet{RW1998}.
A point is critical in the corresponding sense if the
subdifferential contains zero.
All three sets equal $\{\nabla h(\bm{x})\}$ when $h$ is strictly
differentiable at $\bm{x}$, in particular when $h$ is $C^1$ at $\bm{x}$.  We call $h$ subdifferentially regular at $\bm{x}$ if
$\widehat\partial h(\bm{x})=\partial_C h(\bm{x})$.
Every local minimizer $\bm{x}$ of $h$ satisfies $\bm{0}\in\widehat\partial h(\bm{x})$.

A function is semialgebraic if its graph is a finite union
of sets defined by finitely many polynomial equalities and
inequalities. It is coercive if $h(\bm{x})\to+\infty$ as
$\|\bm{x}\|\to\infty$, and positively homogeneous of degree one
if $h(t\bm{x})=t h(\bm{x})$ for every $\bm{x}$ and every $t>0$.
For such a homogeneous function that is $C^1$ away from zero,
$\nabla h(t\bm{x})=\nabla h(\bm{x})$ whenever $\bm{x}\ne\bm{0}$ and $t>0$.

\subsection{Main results}
 For a globally Lipschitz function $f\colon\R^d\to\R$ and a linear
tilt $f_{\bm{v}}(\bm{x})=f(\bm{x})-\langle \bm{v},\bm{x}\rangle$, fix a Borel measurable
selection $\bm{g}$ of $\partial_C f$.
Let $\alpha_k>0$ be the stepsize and $\bm{\nu}_k\in\R^d$ a
perturbation vector at iteration $k$. We consider
\begin{equation}\label{eq:basic-update}
 \bm{x}_{k+1}=\bm{x}_k-\alpha_k\bigl(\bm{g}(\bm{x}_k)-\bm{v}\bigr)+\alpha_k\bm{\nu}_k,
 \qquad \bm{g}(\bm{x})\in\partial_C f(\bm{x}).
\end{equation}
A point $\overline{\bm{x}}$ is Clarke critical for $f_{\bm{v}}$ if
$\bm{v}\in\partial_C f(\overline{\bm{x}})$. 
\begin{theorem}
\label{thm:main}
For every $r>0$ and $\gamma\in(\tfrac{1}{2},1]$, there exist a globally
Lipschitz, coercive, semialgebraic function $f\colon\R^2\to\R$,
constants $K,c,C,\kappa>0$, and nonempty open sets
$U,V\subset\R^2$ such that the following statements hold.
\begin{enumerate}
\item[(i)]
With $\alpha_k=(K+k)^{-\gamma}$, for every $\bm{v}\in V$,
$\bm{x}_0\in U$, and every perturbation sequence satisfying
$\|\bm{\nu}_k\|\le r$ for all $k$, the iterates generated
by~\eqref{eq:basic-update} satisfy
\[
    c\alpha_k\le \|\bm{x}_k\|\le C\alpha_k\to0.
\]
Moreover, $f$ is $C^\infty$ in a neighborhood of every $\bm{x}_k$,
and $\|\nabla f_{\bm{v}}(\bm{x}_k)\|\ge\kappa$.

\item[(ii)]
For every $\bm{v}\in V$, the tilted function $f_{\bm{v}}$ is coercive and
\[
    \bm{0}\in\partial_C f_{\bm{v}}(\bm{0}),
    \qquad
    \operatorname{dist}\bigl(\bm{0},\partial f_{\bm{v}}(\bm{0})\bigr)
    \ge\kappa.
\]
In particular, the origin is Clarke critical for $f_{\bm{v}}$ but
neither limiting critical nor a local minimizer.
\end{enumerate}
The constants $c,C,\kappa$ are uniform over all
$\bm{v}\in V$, $\bm{x}_0\in U$, and perturbation sequences with
$\|\bm{\nu}_k\|\le r$.
\end{theorem}
The conclusion holds for every perturbation path, without independence
or centering. It also covers every trajectory of the inexact method
$\bm{x}_{k+1}\in \bm{x}_k-\alpha_k[\partial_C f_{\bm{v}}(\bm{x}_k)+\overline{\mathbb B}_r(\bm{0})]$
of \citet{BLMP2025} started in $U$, with the stepsizes
in~(i). The objective $f$ depends only on $r$; the constant $K$ and
the set $U$ also depend on $\gamma$.

In particular, Theorem~\ref{thm:main} applies to i.i.d.\ uniform
noise, which is the noise model of \citet{DDJ2023}. This gives the
following consequence. 

\begin{corollary}
\label{cor:ddj}
Fix $r>0$ and $\gamma\in(\tfrac12,1)$, and let $f,K,U,V$ be
given by Theorem~\ref{thm:main}.
With $\alpha_k=(K+k)^{-\gamma}$ and
$\bm{\nu}_k\overset{\mathrm{i.i.d.}}{\sim}\mathrm{Unif}(\mathbb{B}_r(\bm{0}))$, we have
\begin{equation}\label{eq:almost-sure-attraction}
    \mathbb{P}\{\bm{x}_k\to\bm{0}\}=1
    \qquad\text{for every }(\bm{v},\bm{x}_0)\in V\times U,
\end{equation}
where $\bm{0}\notin\partial f_{\bm{v}}(\bm{0})$ for every $\bm{v}\in V$.
In particular, for every $\bm{v}\in V$, the iterates converge almost
   surely to a Clarke critical point that is not a local minimizer,
   which is the second alternative in Corollary~5.6 of \citet{DDJ2023}.
\end{corollary} 
\begin{proof}
Fix $(\bm{v},\bm{x}_0)\in V\times U$. Since $\|\bm{\nu}_k\|<r$
for every $k$, Theorem~\ref{thm:main} gives $\bm{x}_k\to\bm{0}$
and $\sup_k\|\bm{x}_k\|\le CK^{-\gamma}$. This proves
\eqref{eq:almost-sure-attraction}, and the boundedness event
in Corollary~5.6 of \citet{DDJ2023} has probability one.
It remains to verify the assumptions of that corollary.
The noise satisfies Assumption~F of \citet{DDJ2023}. With
$\mathcal F_k:=\sigma(\bm{x}_0,\bm{\nu}_0,\ldots,\bm{\nu}_{k-1})$,
it is a martingale difference sequence with
$\mathbb{E}[\|\bm{\nu}_k\|^4\mid\mathcal F_k]=\tfrac{r^4}{3}$, which gives
item~(E3) of Assumption~E in \citet{DDJ2023}.
The bounds $(K+1)^{-\gamma}k^{-\gamma}\le\alpha_k\le k^{-\gamma}$
for $k\ge1$ give~(E2), with $\gamma$ in the range of Theorem~5.2 in \citet{DDJ2023}. Items~(E1) and~(E4) hold because
the selection is Borel measurable and $f$ is finite everywhere.
Finally, $V$ is nonempty and open, so it has positive Lebesgue
measure and is not contained in the exceptional null set of
tilts in Corollary~5.6 of \citet{DDJ2023}.
\end{proof}
Thus subdifferential regularity cannot be removed from the
generic convergence guarantee of \citet{DDJ2023}.  

Moreover, the same objective separates the discrete method from
its continuous-time counterpart. 

\begin{theorem}[The attraction is purely discrete]
\label{thm:flow-separation}
For every $r>0$ and $\gamma\in(\tfrac12,1]$, the function,
constants, and open sets in Theorem~\ref{thm:main} can be chosen
with $\bm{0}\in V$ so that, for every
$(\bm{v},\bm{x}_0)\in V\times U$, the Clarke subgradient flow
\[
  \dot{\bm{x}}(t)\in-\partial_C f_{\bm{v}}(\bm{x}(t)),
  \qquad \bm{x}(0)=\bm{x}_0,
\]
has a unique absolutely continuous solution on $[0,\infty)$. This solution
avoids the origin and converges to a strict local minimizer
of~$f_{\bm{v}}$, whereas, by Theorem~\ref{thm:main}, the iterates
of~\eqref{eq:basic-update} with the same $\bm{v}$ and $\bm{x}_0$
converge to the origin for every perturbation sequence with
$\|\bm{\nu}_k\|\le r$.
\end{theorem}

\paragraph{Proof overview of Theorem~\ref{thm:main}.}
The proof has three steps. 
\begin{enumerate}
\item[(i)] \emph{Stability of the rescaled iteration.}
To obtain $\|\bm{x}_k\|=\Theta(\alpha_k)$, we keep
$\bm{z}_k:=\bm{x}_k/\alpha_k$ bounded and bounded away from zero.
For a positively homogeneous function $F$ that is $C^2$ away from
zero, the gradient is constant along rays, so $\bm{z}_k$ follows
the map $T:=\mathcal I-\nabla F$ up to small errors.
Proposition~\ref{prop:cycle-principle} shows that if $T$ has a cycle
of nonzero points $\bm{p}_0,\ldots,\bm{p}_{m-1}$ with
$\|D(T^m)(\bm{p}_0)\|<1$, then nearby rescaled iterates remain near
the cycle under small bounded perturbations and stepsize ratios
close to one; see Section~\ref{sec:cycle-principle}.

\item[(ii)] \emph{Construction of the homogeneous function.}
We then seek a globally Lipschitz, semialgebraic, positively
homogeneous function $f_0\colon\R^2\to\R$, smooth away from zero,
whose rescaled map has such a cycle. To ensure the required
stationarity properties, we also impose
\[
 \bm{0}\in\operatorname{int}\partial_C f_0(\bm{0}),\qquad
 \widehat\partial f_0(\bm{0})=\emptyset,\qquad
 \inf_{\|\bm{u}\|=1}\|\nabla f_0(\bm{u})\|>0.
\]
These conditions make the origin Clarke critical but not limiting
critical under all sufficiently small linear tilts, and give a
uniform positive gradient bound away from zero.
Section~\ref{sec:angular-profile} constructs $f_0$ and verifies these
properties, including the cycle condition, in
Proposition~\ref{lem:trap-geometry}.

\item[(iii)] \emph{Scaling and coercivity.}
The stability estimate allows only small perturbations. We
therefore scale $f_0$ so that the prescribed noise and the tilts
in $V$ become small in the rescaled iteration. To obtain
coercivity, we then modify the scaled function outside a ball
containing all trapped iterates so that every tilted objective
$f_{\bm{v}}$ with $\bm{v}\in V$ grows to infinity. This modification
preserves global Lipschitz continuity and semialgebraicity, as well
as the trajectory and the local geometry at zero; see
Section~\ref{sec:scaling-coercivity}.
\end{enumerate}

\paragraph{Proof overview of Theorem~\ref{thm:flow-separation}.}
After rescaling, each discrete step moves the iterate from one
cycle direction to the next. The flow, by contrast, changes its
direction continuously, so it must pass through the directions
between consecutive cycle directions, where the objective decreases
along rays. The construction creates these descent directions while keeping
the contracting cycle intact.  Once the flow reaches such a direction, it moves away
from the origin and converges to a strict local minimizer created
by the modification from step~(iii); see
Section~\ref{sec:flow-comparison}.

\section{Stability of the Rescaled Iteration}
\label{sec:cycle-principle}

We begin with the stability criterion.
Let $\bm{z}_k:=\bm{x}_k/\alpha_k$. The desired bounds
$c\alpha_k\le\|\bm{x}_k\|\le C\alpha_k$ are equivalent to
$c\le\|\bm{z}_k\|\le C$.
Thus, our task is to keep the rescaled iterates in a
bounded set away from zero.

Let $F\colon\R^d\to\R$ be positively homogeneous of degree one
and $C^2$ on $\R^d\setminus\{\bm{0}\}$.
For positive stepsizes $\alpha_k$, consider
\[
    \bm{x}_{k+1}=\bm{x}_k-\alpha_k\nabla F(\bm{x}_k)+\alpha_k \bm{w}_k,
\]
where $\bm{w}_k$ represents a perturbation.
Since $\nabla F(t\bm{z})=\nabla F(\bm{z})$ for $t>0$, the rescaled
iteration takes the form
\[
    \bm{z}_{k+1}=\rho_k\bigl(\mathcal T(\bm{z}_k)+\bm{w}_k\bigr),
    \qquad
    \mathcal T(\bm{z}):=\bm{z}-\nabla F(\bm{z}),
    \qquad
    \rho_k:=\frac{\alpha_k}{\alpha_{k+1}}.
\]

When $\rho_k=1$ and $\bm{w}_k=\bm{0}$, the rescaled update reduces to
$\bm{z}_{k+1}=\mathcal T(\bm{z}_k)$.
Suppose that $\mathcal T$ has a cycle of nonzero points
$\bm{p}_0,\ldots,\bm{p}_{m-1}$ satisfying
$\|D(\mathcal T^m)(\bm{p}_0)\|<1$.
Then the $m$-step map $\mathcal T^m$ maps a sufficiently small
closed ball around $\bm{p}_0$ into its interior.
This inclusion persists when $|\rho_k-1|$ and $\|\bm{w}_k\|$
are uniformly small, while the intermediate iterates remain
near the corresponding cycle points.
Thus, if $\bm{z}_0$ lies in this ball, the rescaled iterates
remain bounded and bounded away from zero, yielding
$c\alpha_k\le\|\bm{x}_k\|\le C\alpha_k$.

The following proposition formalizes this argument.
Its proof uses the classical local contraction argument
for fixed-point iterations; see \citet[Section~10.1]{OR1970}.
We include the perturbation estimates needed here.

\begin{proposition}
\label{prop:cycle-principle}
Let $F\colon\R^d\to\R$ be positively homogeneous of degree one
and $C^2$ away from zero.
Suppose $\mathcal T=\mathcal{I}-\nabla F$ has a cycle
$\bm{p}_0,\ldots,\bm{p}_{m-1}\in\R^d\setminus\{\bm{0}\}$ with
$\|D(\mathcal T^m)(\bm{p}_0)\|<1$.
Then there exist $\delta,\varepsilon,c,C>0$ such that,
for positive stepsizes $\alpha_k\to0$ satisfying
\[
    \left|\frac{\alpha_k}{\alpha_{k+1}}-1\right|\le\varepsilon,
    \qquad \|\bm{w}_k\|\le\varepsilon
    \qquad (k\ge0),
\]
the iteration
\[
    \bm{x}_{k+1}=\bm{x}_k-\alpha_k\nabla F(\bm{x}_k)+\alpha_k \bm{w}_k,
    \qquad \bm{x}_0\in\alpha_0 \mathbb B_\delta(\bm{p}_0),
\]
is well defined and satisfies
$c\alpha_k\le\|\bm{x}_k\|\le C\alpha_k$ for all $k\ge0$.
\end{proposition}

\begin{proof}
Choose $\lambda\in(\|D(\mathcal T^m)(\bm{p}_0)\|,1)$.
Since the cycle avoids zero and $\mathcal T$ is $C^1$
away from zero, we may choose $\delta>0$ such that
$\mathbb B:=\overline{\mathbb B}_\delta(\bm{p}_0)$ satisfies
\[
    \|\mathcal T^m(\bm{z})-\bm{p}_0\|
    \le\lambda\|\bm{z}-\bm{p}_0\|,
    \qquad \bm{z}\in\mathbb B,
\]
and $\mathcal T^j(\mathbb B)$ is well defined and avoids
zero for $0\le j\le m$.

Define $\mathcal T_{\rho,\bm{w}}(\bm{z}):=\rho(\mathcal T(\bm{z})+\bm{w})$
and write
\[
    \Phi_{\bm{\theta}}
    :=\mathcal T_{\rho_{m-1},\bm{w}_{m-1}}\circ\cdots
      \circ\mathcal T_{\rho_0,\bm{w}_0},
    \qquad
    \bm{\theta}:=((\rho_i-1,\bm{w}_i))_{i=0}^{m-1}.
\]
At $\bm{\theta}=\bm{0}$, we have $\Phi_{\bm{0}}=\mathcal T^m$.
Since the unperturbed intermediate iterates avoid zero,
these compositions are well defined and jointly continuous
in $(\bm{z},\bm{\theta})$ on a neighborhood of $\mathbb B\times\{\bm{0}\}$.
Compactness of $\mathbb B$ therefore gives
\[
    \sup_{\bm{z}\in\mathbb B}
    \|\Phi_{\bm{\theta}}(\bm{z})-\mathcal T^m(\bm{z})\|\rightarrow0
    \qquad\text{as }\bm{\theta}\to\bm{0}.
\]
Hence we may choose $\varepsilon>0$ such that
$|\rho_i-1|\le\varepsilon$ and $\|\bm{w}_i\|\le\varepsilon$
for $0\le i<m$ imply
\[
    \sup_{\bm{z}\in\mathbb B}
    \|\Phi_{\bm{\theta}}(\bm{z})-\mathcal T^m(\bm{z})\|
    \le\frac{1-\lambda}{2}\delta.
\]
By reducing $\varepsilon$ if necessary, all intermediate
iterates lie in a common compact set
$\mathbb{Q}\subset\R^d\setminus\{\bm{0}\}$ containing $\mathbb B$.
Consequently,
\[
    \|\Phi_{\bm{\theta}}(\bm{z})-\bm{p}_0\|
    \le\|\Phi_{\bm{\theta}}(\bm{z})-\mathcal T^m(\bm{z})\|
       +\|\mathcal T^m(\bm{z})-\bm{p}_0\|
    \le\frac{1+\lambda}{2}\delta<\delta,
    \qquad \bm{z}\in\mathbb B.
\]
Thus every such $m$-step composition maps
$\mathbb B$ into itself.

Starting from $\bm{z}_0=\bm{x}_0/\alpha_0\in\mathbb B$, consider
\[
    \bm{z}_{k+1}
    =\frac{\alpha_k}{\alpha_{k+1}}
      \bigl(\mathcal T(\bm{z}_k)+\bm{w}_k\bigr).
\]
Applying the preceding inclusion to successive blocks
of $m$ steps gives $\bm{z}_{jm}\in\mathbb B$ and $\bm{z}_k\in \mathbb{Q}$
for all $j,k\ge0$.
By homogeneity, $\bm{x}_k:=\alpha_k\bm{z}_k$ satisfies the stated
iteration. Taking
$c:=\min_{\bm{z}\in \mathbb{Q}}\|\bm{z}\|>0$ and
$C:=\max_{\bm{z}\in \mathbb{Q}}\|\bm{z}\|$
proves the bounds and ensures that every gradient
evaluation is well defined.
\end{proof}

\section{A semialgebraic counterexample}
\label{sec:bounded-trap}

We now construct a positively homogeneous semialgebraic
function on $\R^2$ whose origin is Clarke critical
but not limiting critical, while its rescaled map
has a contracting cycle.
Proposition~\ref{prop:cycle-principle} then keeps the rescaled
iterates near the cycle under sufficiently small bounded
perturbations.
Scaling and a radial correction yield the globally
Lipschitz, coercive example in Theorem~\ref{thm:main}.
Finally, we analyze the Clarke subgradient flow of the same
objective and prove Theorem~\ref{thm:flow-separation}.

\subsection{Construction of the homogeneous function}
\label{sec:angular-profile} 
Let
\[
    \bm{p}_0:=(1,0),\quad \bm{p}_1:=(0,1),\quad
    \bm{p}_2:=(-1,0),\quad \bm{p}_3:=(0,-1),
\]
with $\bm{p}_{i+4}=\bm{p}_i$.
We construct $f_0$ so that $\mathcal T:=\mathcal{I}-\nabla f_0$ maps
$\bm{p}_0\mapsto \bm{p}_1\mapsto \bm{p}_2\mapsto \bm{p}_3\mapsto \bm{p}_0$
and satisfies $\|D(\mathcal T^4)(\bm{p}_0)\|<1$, while the origin is
Clarke critical but not limiting critical. 

We specify the values of $f_0$ on the unit circle and extend
linearly along each ray. Write $\bm{u}(\theta):=(\cos\theta,\sin\theta)$
and define
\begin{equation}\label{eq:trap-polar}
 \begin{aligned}
 f_0(\bm{0})&:=0,\qquad f_0({s}\bm{u}(\theta)):={s}\cdot a(\theta)\quad({s}>0),\\
 a(\theta)&:=1-\tfrac14\sin(4\theta)-4\sin^4(2\theta).
 \end{aligned}
\end{equation}
For each fixed angle $\theta$, the function value
$f_0({s}\bm{u}(\theta))$ varies linearly with the radius ${s}$.
As we move away from the origin in this direction, $f_0$ increases
if $a(\theta)>0$ and decreases if $a(\theta)<0$.

At the coordinate angles $\theta_i:=\tfrac{i\pi}{2}$, we have
$\bm{u}(\theta_i)=\bm{p}_i$.
{Along this cycle, the requirements of Proposition~\ref{prop:cycle-principle} reduce to conditions on $a$ at these angles: by~\eqref{eq:trap-gradient-hessian} below, $\mathcal T(\bm{p}_i)=\bm{p}_{i+1}$ holds if and only if $a(\theta_i)=1$ and $a'(\theta_i)=-1$, and then $\|D(\mathcal T^4)(\bm{p}_0)\|<1$ if and only if $|a''(\theta_1)a''(\theta_3)|<1$ and $|a''(\theta_0)a''(\theta_2)|<1$.}
At each coordinate angle $\theta_i$, the correction
$-4\sin^4(2\theta)$ and its first two derivatives are zero.
Thus adding this correction to $1-\tfrac14\sin(4\theta)$
preserves the identities $ a(\theta_i)=1, a'(\theta_i)=-1,  a''(\theta_i)=0$. 
These identities imply
$\mathcal T(\bm p_i)=\bm p_{i+1}$ and
$D(\mathcal T^4)(\bm p_0)=0$, as verified below.
At the angles $\theta=\tfrac{\pi}{4}+\tfrac{i\pi}{2}$, $i=0,1,2,3$,
the correction equals $-4$, giving $a(\theta)=-3$.
Thus, as we move away from the origin, $f_0$ increases
along the coordinate rays and decreases in these four
directions.
In particular, the origin is not a local minimizer.
Figure~\ref{fig:angular-construction} shows the graph
of $a$, the surface of $f_0$, and the four-cycle.

 
The same correction also governs the continuous-time behavior.
Between consecutive coordinate directions, it creates a strict
local minimum of $a$ with negative value. From initial directions
near a coordinate direction, the angle of the continuous flow
approaches this minimum and the radius eventually increases, whereas
the rescaled discrete iterates remain near the four-cycle.
The radial correction in Section~\ref{sec:scaling-coercivity} stops
the radial motion at a circle on which the flow converges to a
strict local minimizer; see Theorem~\ref{thm:flow-separation}.
 
The equivalent Cartesian expression
\begin{equation}\label{eq:trap-function}
 f_0(\bm{x})=\|\bm{x}\|
 -\frac{x_1x_2(x_1^2-x_2^2)}{\|\bm{x}\|^3}
 -\frac{64x_1^4x_2^4}{\|\bm{x}\|^7},\qquad \bm{x}\ne\bm{0},
\end{equation}
makes semialgebraicity explicit.
The following proposition verifies the cycle condition and shows
that the origin is Clarke critical but not limiting critical.
\begin{figure}[t]
\centering
\resizebox{\linewidth}{!}{%
\begin{tikzpicture}[font=\small,>=Stealth]
 
\definecolor{profileblue}{RGB}{32,85,125}
\definecolor{profileorange}{RGB}{190,95,61}
 
\pgfplotsset{
  colormap={profilemap}{
    rgb(0cm)=(0.38,0.56,0.72);
    rgb(3cm)=(0.96,0.96,0.94);
    rgb(4.15cm)=(0.84,0.53,0.39)
  }
}
 
\begin{axis}[
  at={(0cm,0cm)},
  anchor=south west,
  width=6.5cm,
  height=5.2cm,
  scale only axis,
  xmin=0,
  xmax=1.570796327,
  ymin=-3.35,
  ymax=1.5,
  axis lines=left,
  axis line style={black!60},
  tick style={black!55},
  xtick={0,0.7853981634,1.570796327},
  xticklabels={$0$,$\tfrac{\pi}{4}$,$\tfrac{\pi}{2}$},
  ytick={-3,-2,-1,0,1},
  xlabel={$\theta$},
  ylabel={$a(\theta)$},
  xlabel style={
    at={(axis description cs:1.025,0)},
    anchor=west
  },
  ylabel style={
    rotate=-90,
    at={(axis description cs:0,1.015)},
    anchor=south
  },
  ymajorgrids,
  grid style={black!9},
  tick label style={font=\small},
  legend style={
    draw=none,
    fill=none,
    font=\small,
    at={(0.5,-0.18)},
    anchor=north,
    legend columns=2,
    column sep=12pt
  },
  clip=false
]
 
\addplot[
  black!45,
  dashed,
  line width=0.8pt,
  domain=0:pi/2,
  samples=241
]
  {1-0.25*sin(deg(4*x))};
\addlegendentry{$1-\tfrac14\sin(4\theta)$}
 
\addplot[
  profileblue,
  line width=1.15pt,
  domain=0:pi/2,
  samples=321
]
  {1-0.25*sin(deg(4*x))-4*(sin(deg(2*x)))^4};
\addlegendentry{$a(\theta)$}
 
\addplot[
  only marks,
  mark=*,
  mark size=2pt,
  profileblue,
  forget plot
]
  coordinates {(0,1) (1.570796327,1)};
 
\addplot[
  only marks,
  mark=*,
  mark size=2pt,
  profileorange,
  forget plot
]
  coordinates {(0.7853981634,-3)};
 
\end{axis}
 
\node[anchor=west,font=\normalsize]
  at (-0.1cm,6cm) {(a) Angular profile};
 
\begin{axis}[
  at={(6.5cm,-0.8cm)},
  anchor=south west,
  width=10cm,
  height=8cm,
  view={-55}{27},
  hide axis,
  xmin=-1.12,
  xmax=1.12,
  ymin=-1.12,
  ymax=1.12,
  zmin=-3.2,
  zmax=1.3,
  colormap name=profilemap,
  point meta min=-3,
  point meta max=1.15,
  colorbar right,
  colorbar style={
    hide axis=false,
    at={(1.02,0.5)},
    anchor=west,
    width=0.20cm,
    height=4.5cm,
    title={$f_0(\bm{x})$},
    ytick={-3,-2,-1,0,1},
    tick label style={font=\small},
    title style={font=\small},
    scaled ticks=false
  },
  z buffer=sort,
  clip=false
]
 
\addplot3[
  surf,
  shader=interp,
  draw=none,
  domain=0:360,
  y domain=0:1,
  samples=101,
  samples y=18
]
  ({y*cos(x)},
   {y*sin(x)},
   {y*(1-0.25*sin(4*x)-4*(sin(2*x))^4)});
 
\addplot3[
  black!50,
  densely dotted,
  thin,
  z buffer=none
]
  coordinates {(1,0,1) (0,0,0) (-1,0,1)};
 
\addplot3[
  black!50,
  densely dotted,
  thin,
  z buffer=none
]
  coordinates {(0,1,1) (0,0,0) (0,-1,1)};
 
\addplot3[
  profileblue,dashed,line width=0.8pt,
  ->,z buffer=none
]
  coordinates {(1,0,1) (0,1,1)};
 
\addplot3[
  profileblue,dashed,line width=0.8pt,
  ->,z buffer=none
]
  coordinates {(0,1,1) (-1,0,1)};
 
\addplot3[
  profileblue,dashed,line width=0.8pt,
  ->,z buffer=none
]
  coordinates {(-1,0,1) (0,-1,1)};
 
\addplot3[
  profileblue,dashed,line width=0.8pt,
  ->,z buffer=none
]
  coordinates {(0,-1,1) (1,0,1)};
 
\addplot3[
  only marks,
  mark=*,
  mark size=1.9pt,
  profileblue
]
  coordinates {
    (1,0,1)
    (0,1,1)
    (-1,0,1)
    (0,-1,1)
  };
 
\addplot3[
  only marks,
  mark=*,
  mark size=1.5pt,
  black
]
  coordinates {(0,0,0)};
 
\node[above right,inner sep=3pt]
  at (axis cs:1,0,1) {$\bm{p}_0$};
\node[above,inner sep=3pt]
  at (axis cs:0,1,1) {$\bm{p}_1$};
\node[above left,inner sep=3pt]
  at (axis cs:-1,0,1) {$\bm{p}_2$};
\node[below right,inner sep=3pt]
  at (axis cs:0,-1,1) {$\bm{p}_3$};
\node[below,inner sep=5pt]
  at (axis cs:0,0,0) {$f_0(\bm{0})=0$};
 
\end{axis}
 
\node[anchor=west,font=\normalsize]
  at (7.9cm,6cm) {(b) Surface and four-cycle};
 
\end{tikzpicture}%
}
\caption{
Construction of the homogeneous function $f_0$.
(a) The correction $-4\sin^4(2\theta)$ makes the diagonal
value negative while preserving the value and first two
derivatives at the coordinate angles.
(b) The surface $f_0({s}\bm{u}(\theta))={s} a(\theta)$ over the unit
disk; the color bar gives the function values.
Dashed arrows show the cycle $\mathcal{T}(\,\bm{p}_i)=\bm{p}_{i+1}$,
with each point lifted to height $f_0(\bm{p}_i)=1$.
The arrows indicate discrete updates, not paths along
the surface.
}
\label{fig:angular-construction}
\end{figure}
\begin{proposition}
\label{lem:trap-geometry}
The function $f_0$ is semialgebraic, positively homogeneous
of degree one, globally $18$-Lipschitz, and $C^\infty$
away from zero.
Let $P_i:=\bm{p}_i\bm{p}_i^{\mathsf T}$.
\begin{enumerate}
\item[(i)]
For $i=0,1,2,3$, with indices read modulo four,
\begin{equation}\label{eq:trap-cycle-jets}
\mathcal T(\bm{p}_i)=\bm{p}_{i+1},\qquad D\mathcal T(\bm{p}_i)=P_i,\qquad
 D(\mathcal T^2)(\bm{p}_i)=0.
\end{equation}
In particular, $D(\mathcal T^4)(\bm{p}_0)=0$.
 
\item[(ii)]
We have
$
 [-1,1]\times[-1,1]\subseteq\partial_C f_0(\bm{0})$, $
  \widehat\partial f_0(\bm{0})=\varnothing$
and
\begin{equation}\label{eq:trap-gradient-lower}
 \inf_{\|\bm{u}\|=1}\|\nabla f_0(\bm{u})\|\ge\frac12.
\end{equation}
Consequently,
\[
 \bm{0}\in\operatorname{int}\partial_C f_0(\bm{0}),\qquad
 \operatorname{dist}(\bm{0},\partial f_0(\bm{0}))\ge\frac12.
\]
\end{enumerate}
\end{proposition}

\begin{proof}
The Cartesian formula~\eqref{eq:trap-function} shows that
$f_0$ is semialgebraic and smooth away from zero.
Positive homogeneity follows from~\eqref{eq:trap-polar}.
Moreover, $|a(\theta)|\le \tfrac{13}{4}$ for every $\theta$, so
\[
 |f_0(\bm{x})|\le\frac{13}{4}\|\bm{x}\|\rightarrow0
 \qquad\text{as }\bm{x}\to\bm{0}.
\]
Since $f_0(\bm{0})=0$, this proves continuity at the origin.
 
Let $\bm{t}(\theta):=(-\sin\theta,\cos\theta)$ denote the unit
tangent vector to the unit circle. Differentiating the
polar formula gives
\begin{equation}\label{eq:trap-gradient-hessian}
 \begin{aligned}
 \nabla f_0({s}\bm{u}(\theta))
   &=a(\theta)\bm{u}(\theta)+a'(\theta)\bm{t}(\theta),\\
 \nabla^2 f_0({s}\bm{u}(\theta))
   &=\frac{a(\theta)+a''(\theta)}{{s}}
     \bm{t}(\theta)\bm{t}(\theta)^{\mathsf T}.
 \end{aligned}
\end{equation}
The derivative
\[
 a'(\theta)=-\cos(4\theta)
     -8(1-\cos(4\theta))\sin(4\theta)
\]
satisfies $|a'(\theta)|\le17$.
Since $\bm{u}(\theta)$ and $\bm{t}(\theta)$ are orthonormal,
\[
 \|\nabla f_0({s}\bm{u}(\theta))\|
 =\sqrt{a(\theta)^2+a'(\theta)^2}
 \le\sqrt{\left(\tfrac{13}{4}\right)^2+17^2}<18.
\]
For any segment avoiding zero, integration of this gradient
bound gives $|f_0(\bm{x})-f_0(\bm{y})|\le18\|\bm{x}-\bm{y}\|$.
If the segment contains zero, split it at zero and use
continuity there to obtain the same bound.
Thus $f_0$ is globally $18$-Lipschitz.
 
\emph{(i) }
At $\theta_i:=\tfrac{i\pi}{2}$, the angular profile satisfies
\[
 a(\theta_i)=1,\qquad a'(\theta_i)=-1,
 \qquad a''(\theta_i)=0.
\]
Since $\bm{u}(\theta_i)=\bm{p}_i$ and $\bm{t}(\theta_i)=\bm{p}_{i+1}$,
substituting ${s}=1$ into~\eqref{eq:trap-gradient-hessian} gives
\[
 \nabla f_0(\bm{p}_i)=\bm{p}_i-\bm{p}_{i+1},
 \qquad \nabla^2 f_0(\bm{p}_i)=\bm{p}_{i+1}\bm{p}_{i+1}^{\mathsf T}.
\]
The gradient identity yields
$\mathcal T(\bm{p}_i)=\bm{p}_i-\nabla f_0(\bm{p}_i)=\bm{p}_{i+1}$, so the four points
form a cycle.
To verify contraction, note that $\bm{p}_i,\bm{p}_{i+1}$ form an
orthonormal basis of $\R^2$. Hence
\[
 D\mathcal T(\bm{p}_i)=\mathcal{I} -\nabla^2 f_0(\bm{p}_i)
        =\mathcal{I}-\bm{p}_{i+1}\bm{p}_{i+1}^{\mathsf T}
        =\bm{p}_i\bm{p}_i^{\mathsf T}=P_i.
\]
By the chain rule,
\[
 D(\mathcal{T}^2)(\bm{p}_i)
 =D\mathcal{T}(\bm{p}_{i+1})D\mathcal{T}(\bm{p}_i)
 =P_{i+1}P_i
 =\bm{p}_{i+1}(\bm{p}_{i+1}^{\mathsf T}\bm{p}_i)\bm{p}_i^{\mathsf T}
 =0.
\]
Since $\mathcal{T}^2(\bm{p}_0)=\bm{p}_2$, another application of the chain rule gives
\[
 D(\mathcal{T}^4)(\bm{p}_0)
 =D(\mathcal{T}^2)(\bm{p}_2)D(\mathcal{T}^2)(\bm{p}_0)=0.
\]
Thus the return map $\mathcal{T}^4$ satisfies the contraction condition
$\|D(\mathcal{T}^4)(\bm{p}_0)\|<1$.
 
\emph{(ii)}
We first show that $[-1,1]\times[-1,1]\subseteq\partial_C f_0(\bm{0})$.
For each $i=0,1,2,3$, the nonzero points $\bm{p}_i/n$ converge
to zero as $n\to\infty$. By homogeneity,
we have 
 $\nabla f_0(\bm{p}_i/n)=\nabla f_0(\bm{p}_i)=\bm{p}_i-\bm{p}_{i+1}$.
Thus each of the four vectors
\[
 \begin{aligned}
 \bm{p}_0-\bm{p}_1&=(1,-1), & \bm{p}_1-\bm{p}_2&=(1,1),\\
 \bm{p}_2-\bm{p}_3&=(-1,1), & \bm{p}_3-\bm{p}_0&=(-1,-1)
 \end{aligned}
\]
is a gradient limit at zero and hence belongs to
$\partial f_0(\bm{0})$.
Since $\partial_C f_0(\bm{0})$ is the convex hull of
$\partial f_0(\bm{0})$, it contains
\[
 \operatorname{conv}\{(1,-1),(1,1),(-1,1),(-1,-1)\}
 =[-1,1]\times[-1,1].
\]
In particular, $\bm{0}\in\operatorname{int}\partial_C f_0(\bm{0})$. 
We next show that $\widehat\partial f_0(\bm{0})=\varnothing$.
Let $\bm{d}:=\tfrac{(1,1)}{\sqrt{2}}$. Since
$a(\tfrac{\pi}{4})=a(\tfrac{5\pi}{4})=-3$, we have
$
 f_0(t\bm{d})=f_0(-t\bm{d})=-3t
$ for any $t>0$. 
Suppose that $\bm{g}\in\widehat\partial f_0(\bm{0})$. By definition,
\[
 f_0(\bm{x})\ge f_0(\bm{0})+\langle \bm{g},\bm{x}\rangle+o(\|\bm{x}\|)
 \qquad\text{as }\bm{x}\to\bm{0}.
\]
Applying this inequality to $\bm{x}=t\bm{d}$ and $\bm{x}=-t\bm{d}$ gives
$-3t\ge t\langle \bm{g},\bm{d}\rangle+o(t)$ and $ -3t\ge-t\langle \bm{g},\bm{d}\rangle+o(t)$.
Dividing by $t>0$ and letting $t\downarrow0$, we obtain $\langle \bm{g},\bm{d}\rangle\le-3$ and $-\langle \bm{g},\bm{d}\rangle\le-3$.
Adding these inequalities yields $0\le-6$, a contradiction.
Thus $\widehat\partial f_0(\bm{0})=\varnothing$.
We now prove that the gradients are uniformly bounded away
from zero. By~\eqref{eq:trap-gradient-hessian},
\[
 \|\nabla f_0(\bm{u}(\theta))\|^2
 =a(\theta)^2+a'(\theta)^2.
\]
It therefore suffices to show that $a(\theta)$ and
$a'(\theta)$ cannot both have absolute value less than $\tfrac12$.
Using $4\sin^4(2\theta)=(1-\cos(4\theta))^2$, we obtain
\[
 \begin{aligned}
 a(\theta)&=2\cos(4\theta)-\cos^2(4\theta)
             -\tfrac14\sin(4\theta),\\
 a'(\theta)&=-\cos(4\theta)
             -8(1-\cos(4\theta))\sin(4\theta).
 \end{aligned}
\]
If $\cos(4\theta)\ge\tfrac12$, then
$a(\theta)\ge\tfrac34-\tfrac14=\tfrac12$.
If $\cos(4\theta)\le-\tfrac12$, then
$a(\theta)\le-\tfrac54+\tfrac14=-1$.
Otherwise, $|\cos(4\theta)|<\tfrac12$, so
$|\sin(4\theta)|\ge\tfrac{\sqrt3}{2}$ and
$1-\cos(4\theta)\ge\tfrac12$. Hence, 
\[
 \begin{aligned}
 |a'(\theta)|
 &\ge8(1-\cos(4\theta))|\sin(4\theta)|
      -|\cos(4\theta)|\ge2\sqrt3-\frac12>\frac12.
 \end{aligned}
\]
These cases prove~\eqref{eq:trap-gradient-lower}.
Since the gradient is constant along each ray,
$\|\nabla f_0(\bm{x})\|\ge\tfrac12$ for $\bm{x}\ne\bm{0}$.
 
Finally, let $\bm{g}\in\partial f_0(\bm{0})$.
By definition, there exist $\bm{x}_n\to\bm{0}$ and $\bm{g}_n\to \bm{g}$ with
$\bm{g}_n\in\widehat\partial f_0(\bm{x}_n)$.
Since $\widehat\partial f_0(\bm{0})=\varnothing$, each $\bm{x}_n$ is
nonzero. Smoothness at $\bm{x}_n$ then gives
$\bm{g}_n=\nabla f_0(\bm{x}_n)$, and hence
\[
 \|\bm{g}\|=\lim_{n\to\infty}\|\nabla f_0(\bm{x}_n)\|
 \ge\frac12.
\]
Thus $\operatorname{dist}(\bm{0},\partial f_0(\bm{0}))
\ge \tfrac{1}{2}$, completing the proof.
\end{proof}

Since $D(\mathcal T^4)(\bm{p}_0)=0$, Proposition~\ref{prop:cycle-principle}
yields $\|\bm{x}_k\|=\Theta(\alpha_k)$ for $\bm{x}_0/\alpha_0$
sufficiently close to $\bm{p}_0$, provided that $\alpha_k\to0$,
the stepsize ratios are sufficiently close to one, and the
perturbations are sufficiently small.
Thus the iterates converge to the origin, which is Clarke
critical but not limiting critical; see Figure~\ref{fig:angular-construction}.

\subsection{Scaling and coercivity}
\label{sec:scaling-coercivity}
 
We complete the construction.
Scaling
makes the prescribed noise and tilts small enough
for Proposition~\ref{prop:cycle-principle} to apply.
A radial correction then ensures coercivity without
changing the iterates.
 
\begin{proof}[Proof of Theorem~\ref{thm:main}]
Fix $r>0$ and $\gamma\in(\tfrac12,1]$.
By Proposition~\ref{lem:trap-geometry}, the function $f_0$
satisfies the assumptions of
Proposition~\ref{prop:cycle-principle}.
Let $\delta,\varepsilon,c_0,C_0>0$ be the constants
provided by that proposition.
 
Choose
$
 M>\max\left\{4,\tfrac{r+1}{\varepsilon}\right\},
$
and define
\begin{equation}\label{eq:coercive-construction}
 f(\bm{x}):=Mf_0(\bm{x})+4M(\|\bm{x}\|-M)_+,
 \qquad V:=\mathbb B_1(\bm{0}),
\end{equation}
where $(t)_+:=\max\{t,0\}$.
Note that the choice of the objective $f$ depends on the perturbation
bound $r$, but not on the stepsize exponent $\gamma$.
Choose $K\ge1$ sufficiently large that
$K^{-1}\le\varepsilon$ and $C_0K^{-\gamma}<1$, and set
\[
 \alpha_k:=(K+k)^{-\gamma},
 \qquad
 U:=M\alpha_0\mathbb B_\delta(\bm{p}_0).
\]
We will prove the theorem with
$c:=Mc_0$, $C:=MC_0$, and $\kappa:=\tfrac M2-1>0$.
 
The function $f$
is semialgebraic and globally $22M$-Lipschitz, since
$f_0$ is semialgebraic and $18$-Lipschitz, and
$\bm{x}\mapsto(\|\bm{x}\|-M)_+$ is semialgebraic and $1$-Lipschitz.
Moreover, $f_0(\bm{x})\ge-\tfrac{13}{4}\|\bm{x}\|$, so for
$\bm{v}\in V$ and $\|\bm{x}\|\ge M$,
\[
 f_{\bm{v}}(\bm{x})
 \ge -\frac{13M}{4}\|\bm{x}\|
      +4M(\|\bm{x}\|-M)-\|\bm{v}\|\|\bm{x}\|
 \ge \left(\frac{3M}{4}-1\right)\|\bm{x}\|-4M^2.
\]
Since $M>4$, every $f_{\bm{v}}$ is coercive.
 
Fix $\bm{v}\in V$, $\bm{x}_0\in U$, and a perturbation sequence
satisfying $\|\bm{\nu}_k\|\le r$ for all $k$.
Consider the auxiliary iteration
\[
 \bm{y}_{k+1}
 =\bm{y}_k-\alpha_k\nabla f_0(\bm{y}_k)
      +\alpha_k\frac{\bm{v}+\bm{\nu}_k}{M},
 \qquad \bm{y}_0:=\frac{\bm{x}_0}{M}.
\]
The choices of $M$, $K$, and $U$ ensure that
\[
 \left\|\frac{\bm{v}+\bm{\nu}_k}{M}\right\|
 \le\frac{1+r}{M}<\varepsilon,
 \qquad
 1\le\frac{\alpha_k}{\alpha_{k+1}}
 =\left(1+\frac{1}{K+k}\right)^\gamma
 \le1+\frac1K\le1+\varepsilon,
\]
and $\bm{y}_0/\alpha_0\in\mathbb B_\delta(\bm{p}_0)$.
Proposition~\ref{prop:cycle-principle} therefore gives
$c_0\alpha_k\le\|\bm{y}_k\|\le C_0\alpha_k$.
Set $\bm{x}_k:=M\bm{y}_k$. Then
\[
 0<Mc_0\alpha_k\le\|\bm{x}_k\|
 \le MC_0\alpha_k\le MC_0K^{-\gamma}<M.
\]
Hence $f=Mf_0$ in a neighborhood of every $\bm{x}_k$, so
$\partial_C f(\bm{x}_k)=\{\nabla f(\bm{x}_k)\}$ and every selection
satisfies $\bm{g}(\bm{x}_k)=\nabla f(\bm{x}_k)$.
By homogeneity, $\nabla f_0(M\bm{y}_k)=\nabla f_0(\bm{y}_k)$,
so multiplying the auxiliary update by $M$ yields
\[
 \bm{x}_{k+1}
 =\bm{x}_k-\alpha_k\bigl(\nabla f(\bm{x}_k)-\bm{v}\bigr)
      +\alpha_k\bm{\nu}_k.
\]
Thus these are precisely the iterates of
\eqref{eq:basic-update}. The function $f$ is $C^\infty$ in a
neighborhood of every $\bm{x}_k$, and the required bounds hold
with $c=Mc_0$ and $C=MC_0$.
 
Finally, the gradient bound for $f_0$ gives
\[
 \|\nabla f_{\bm{v}}(\bm{x}_k)\|
 =\|M\nabla f_0(\bm{x}_k)-\bm{v}\|
 \ge\frac M2-\|\bm{v}\|
 \ge\frac M2-1=\kappa.
\]
Since $f=Mf_0$ near the origin, the subdifferential
calculus gives
\[
 \partial_C f_{\bm{v}}(\bm{0})=M\partial_C f_0(\bm{0})-\bm{v},
 \qquad
 \partial f_{\bm{v}}(\bm{0})
 =M\partial f_0(\bm{0})-\bm{v}.
\]
By Proposition~\ref{lem:trap-geometry},
$[-1,1]\times[-1,1]\subseteq\partial_C f_0(\bm{0})$.
Since $\|\bm{v}\|<1<M$, we have $\bm{v}/M\in[-1,1]\times[-1,1]$,
and hence $\bm{0}\in\partial_C f_{\bm{v}}(\bm{0})$.
On the other hand,
\[
 \operatorname{dist}\bigl(\bm{0},\partial f_{\bm{v}}(\bm{0})\bigr)
 \ge M\operatorname{dist}
       \bigl(\bm{0},\partial f_0(\bm{0})\bigr)-\|\bm{v}\|
 \ge\frac M2-1=\kappa>0.
\]
The necessary condition for local optimality therefore
implies that the origin is not a local minimizer of $f_{\bm{v}}$.
All constants and the nonempty open sets $U,V$ were chosen
independently of $\bm{v}$, $\bm{x}_0$, and the perturbation sequence.
This completes the proof.
\end{proof}

\subsection{Continuous dynamics}
\label{sec:flow-comparison}
 
We now prove Theorem~\ref{thm:flow-separation} for the coercive
objective constructed above. The angular minimum attracts
the flow, and the radial correction makes its limit a strict
local minimizer on the circle $\|\bm{x}\|=M$.
Figure~\ref{fig:paired-trajectories} illustrates the two
trajectories from a common initial point, with zero tilt
and zero perturbations.
\begin{figure}[t]
\centering
\includegraphics[width=\linewidth]{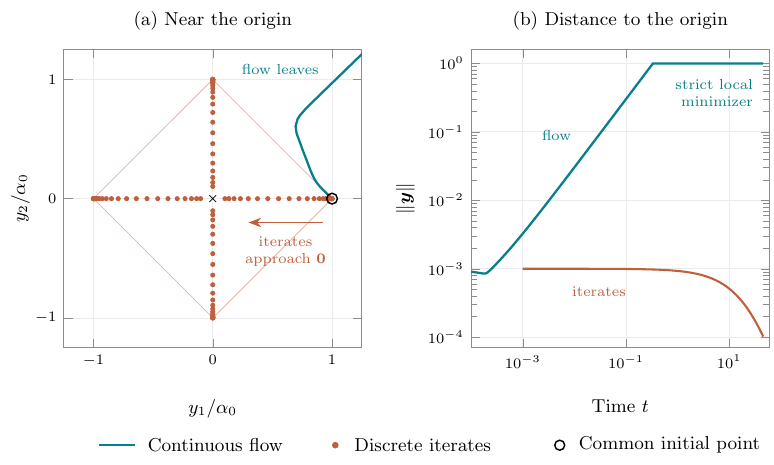}
\caption{Continuous flow and discrete iterates of the untilted
objective from the same initial point $\bm{y}_0=\alpha_0(1,0)$,
without perturbations, in the coordinates $\bm{y}=\bm{x}/M$, with
$\alpha_k=(10^4+k)^{-3/4}$.
(a)~Near the origin, in units of $\alpha_0$. Dots show selected
iterates up to $k=2\cdot10^5$; segments connect the first eight.
Each discrete step moves the iterate to the next coordinate
direction, and the iterates approach the origin (cross), whereas
the flow turns away from the coordinate direction and leaves.
(b)~Distance to the origin versus time. The $k$-th iterate is
placed at time $t_k=\sum_{j<k}\alpha_j$, since an update with
stepsize $\alpha_j$ corresponds to a time increment $\alpha_j$ of the flow. After a brief initial decrease, the distance of the flow
grows until it reaches $\|\bm{y}\|=1$. The radius then remains
constant while the flow converges to a strict local minimizer,
whereas the distance of the iterates decays like $\alpha_k$.
Theorem~\ref{thm:flow-separation} proves this separation on open
sets of initial points and tilts, uniformly over bounded
perturbations.}
\label{fig:paired-trajectories}
\end{figure}
\begin{proof}[Proof of Theorem~\ref{thm:flow-separation}]
Let $f,M,U,V$ be the choices in the proof of
Theorem~\ref{thm:main}. Define the tilted angular profile by
$b_{\bm{v}}(\theta):=a(\theta)-M^{-1}\langle\bm{v},\bm{u}(\theta)\rangle$.
Since $f_0(s\bm{u}(\theta))=s\,a(\theta)$, we have
\[
 f_{\bm{v}}(s\bm{u}(\theta))
 =Ms\,b_{\bm{v}}(\theta)+4M(s-M)_+.
\]

We first identify a common angular interval on which
$b_{\bm{v}}$ has a unique nondegenerate minimum for all
sufficiently small tilts.
Differentiating~\eqref{eq:trap-polar} gives
\[
 a'<0
 \quad\text{on }\left[0,\tfrac{3\pi}{16}\right],
 \qquad
 a''>0
 \quad\text{on }\left[\tfrac{3\pi}{16},\tfrac{\pi}{4}\right],
 \qquad
 a'\left(\tfrac{\pi}{4}\right)=1.
\]
Choose $\eta>0$ such that $a'<0$ on $[-\eta,0]$, and set
$J:=[-\eta,\tfrac{\pi}{4}]$. By compactness, $a'$ is uniformly
negative on $[-\eta,\tfrac{3\pi}{16}]$, and $a''$ is uniformly
positive on $[\tfrac{3\pi}{16},\tfrac{\pi}{4}]$.
Since $|b_{\bm{v}}'-a'|\le\tfrac{\|\bm{v}\|}{M}$ and
$|b_{\bm{v}}''-a''|\le\tfrac{\|\bm{v}\|}{M}$, there is an open neighborhood
$V_0\subseteq V$ of zero such that, for every $\bm{v}\in V_0$,
\[
 b_{\bm{v}}'<0
 \quad\text{on }\left[-\eta,\tfrac{3\pi}{16}\right],
 \qquad
 b_{\bm{v}}''>0
 \quad\text{on }\left[\tfrac{3\pi}{16},\tfrac{\pi}{4}\right],
 \qquad
 b_{\bm{v}}'\left(\tfrac{\pi}{4}\right)>0.
\]
Thus $b_{\bm{v}}$ has a unique critical point $\beta_{\bm{v}}$ on
$J$. It lies in $(\tfrac{3\pi}{16},\tfrac{\pi}{4})$, is the unique
minimizer of $b_{\bm{v}}$ on $J$, and satisfies
$b_{\bm{v}}''(\beta_{\bm{v}})>0$.
Since $\|\bm{v}\|<1$ and $M>4$, we have $|b_{\bm{v}}-a|\le\tfrac14$.
Together with $a\ge-\tfrac{13}{4}$ and $a(\tfrac{\pi}{4})=-3$, this
gives
\begin{equation}\label{eq:b-beta-bounds}
 -4<b_{\bm{v}}(\beta_{\bm{v}})
 \le b_{\bm{v}}\left(\tfrac{\pi}{4}\right)<0.
\end{equation}
Along the ray with angle $\beta_{\bm{v}}$, the objective therefore
strictly decreases for $0<s<M$ and strictly increases for $s>M$.

We next write the flow in polar coordinates.
The point $M\alpha_0\bm{p}_0$ belongs to $U$ and has polar angle
$0\in\operatorname{int}J$. Choose an open ball $U_0\subseteq U$
centered at this point whose polar angles lie in
$\operatorname{int}J$. Fix $\bm{v}\in V_0$ and
$\bm{x}_0=s_0\bm{u}(\theta_0)\in U_0$. Then
$\theta_0\in\operatorname{int}J$, and $0<s_0<M$ by the proof of
Theorem~\ref{thm:main}.
Let $\bm{x}(t)$ be a Clarke subgradient flow from $\bm{x}_0$, and
write $\bm{x}=s\bm{u}(\theta)$ as long as $s>0$.
Away from the origin, the term $Ms\,b_{\bm{v}}(\theta)$ is smooth,
whereas $4M(s-M)_+$ depends only on $s$. Hence every
$\bm{g}\in\partial_C f_{\bm{v}}(s\bm{u}(\theta))$ has tangential
component $\langle\bm{g},\bm{u}'(\theta)\rangle=Mb_{\bm{v}}'(\theta)$.
Its radial component equals $Mb_{\bm{v}}(\theta)$ for $s<M$ and
$M(b_{\bm{v}}(\theta)+4)$ for $s>M$, and lies in
$[Mb_{\bm{v}}(\theta),M(b_{\bm{v}}(\theta)+4)]$ for $s=M$.
Since $\dot{\bm{x}}=\dot s\,\bm{u}(\theta)+s\dot\theta\,\bm{u}'(\theta)$,
taking components in
$\dot{\bm{x}}\in-\partial_C f_{\bm{v}}(\bm{x})$ and changing time by
$\mathrm{d}\tau/\mathrm{d}t=M/s$ give
\begin{equation}\label{eq:polar-flow}
 \frac{\mathrm{d}\theta}{\mathrm{d}\tau}=-b_{\bm{v}}'(\theta),
 \qquad
 \frac{\mathrm{d}s}{\mathrm{d}\tau}\in
 \begin{cases}
 \{-s\,b_{\bm{v}}(\theta)\}, & s<M,\\[2pt]
 \bigl[-M(b_{\bm{v}}(\theta)+4),\,-Mb_{\bm{v}}(\theta)\bigr], & s=M,\\[2pt]
 \{-s\,(b_{\bm{v}}(\theta)+4)\}, & s>M.
 \end{cases}
\end{equation}

We first analyze the angular motion. The angular equation
in~\eqref{eq:polar-flow} no longer involves the radius.
Since $b_{\bm{v}}'<0$ to the left of $\beta_{\bm{v}}$ and
$b_{\bm{v}}'>0$ to its right on $J$, its solution remains in $J$,
moves monotonically toward $\beta_{\bm{v}}$, and converges to
$\beta_{\bm{v}}$. Moreover, $b_{\bm{v}}(\theta(\tau))$ is
nonincreasing, since
$\tfrac{\mathrm{d}}{\mathrm{d}\tau}b_{\bm{v}}(\theta(\tau))
=-|b_{\bm{v}}'(\theta(\tau))|^2$.

We next analyze the radial motion. As long as $s<M$,
\[
 s(\tau)=s_0\exp\left(
   -\int_0^\tau b_{\bm{v}}(\theta(\sigma))\,\mathrm{d}\sigma
 \right),
\]
which is positive for every finite $\tau$.
Since $b_{\bm{v}}(\theta(\tau))\to b_{\bm{v}}(\beta_{\bm{v}})<0$,
the radius reaches $M$ at a finite time $\tau_*$.
Moreover, $b_{\bm{v}}(\theta(\tau_*))<0$; otherwise monotonicity
would give $b_{\bm{v}}(\theta)\ge0$ on $[0,\tau_*]$ and hence
$s(\tau_*)\le s_0<M$. By~\eqref{eq:b-beta-bounds} and monotonicity,
$-4<b_{\bm{v}}(\theta(\tau))<0$ for all $\tau\ge\tau_*$.
The radial velocity is then positive for $s<M$ and negative for
$s>M$, so
$\tfrac{\mathrm{d}}{\mathrm{d}\tau}(s-M)^2\le0$ almost everywhere,
and $s(\tau)=M$ for all $\tau\ge\tau_*$.

We now return to the original time variable. The inverse time
change $t(\tau):=\int_0^\tau 
\tfrac{s(\sigma)}{M}\,\mathrm{d}\sigma$ is
strictly increasing and satisfies $t(\tau)=t(\tau_*)+\tau-\tau_*$
for $\tau\ge\tau_*$, so it maps $[0,\infty)$ onto $[0,\infty)$.
The solution of~\eqref{eq:polar-flow} therefore defines a Clarke
subgradient flow on $[0,\infty)$ that never reaches the origin and
satisfies $\bm{x}(t)\to M\bm{u}(\beta_{\bm{v}})$.
Conversely, every Clarke subgradient flow from $\bm{x}_0$
satisfies~\eqref{eq:polar-flow} while $s>0$. Since the angular
equation has a unique solution and the radial argument above
applies to every solution, the flow is unique.

We finally verify that the limit is a strict local minimizer.
For $\theta$ near $\beta_{\bm{v}}$, continuity
and~\eqref{eq:b-beta-bounds} give $-4<b_{\bm{v}}(\theta)<0$, so
along each such direction the objective strictly decreases for
$s<M$ and strictly increases for $s>M$. Hence, for $(s,\theta)$
near $(M,\beta_{\bm{v}})$,
\[
 f_{\bm{v}}(s\bm{u}(\theta))
 \ge f_{\bm{v}}(M\bm{u}(\theta))
 =M^2b_{\bm{v}}(\theta)
 \ge M^2b_{\bm{v}}(\beta_{\bm{v}})
 =f_{\bm{v}}(M\bm{u}(\beta_{\bm{v}})),
\]
with equality only if $s=M$ and $\theta=\beta_{\bm{v}}$.

Since $U_0\subseteq U$ and $V_0\subseteq V$, the discrete
conclusions of Theorem~\ref{thm:main} remain valid on these sets
with the same constants. Replacing $U,V$ by $U_0,V_0$ completes
the proof.
\end{proof}
\begin{remark}
\label{rem:global-flow}
{The flow conclusion of Theorem~\ref{thm:flow-separation} holds
from almost every initial point. We have $a'>0$ on
$[\tfrac{\pi}{4},\tfrac{7\pi}{16}]$, $a''<0$ on
$[\tfrac{7\pi}{16},\tfrac{15\pi}{32}]$, and $a'<0$ on
$[\tfrac{15\pi}{32},\tfrac{\pi}{2}]$. Together with the analysis on
$[0,\tfrac{\pi}{4}]$ in the proof of Theorem~\ref{thm:flow-separation},
these signs give one nondegenerate minimum and one nondegenerate maximum
in each period of length $\tfrac{\pi}{2}$. Since $b_{\bm{v}}\to a$ in
$C^2$ as $\bm{v}\to\bm{0}$, we may shrink $V$ so that, for every
$\bm{v}\in V$, $b_{\bm{v}}$ has exactly four local minimizers and four
local maximizers on the circle, with $b_{\bm{v}}''>0$ and
$-4<b_{\bm{v}}<0$ at each minimizer as in~\eqref{eq:b-beta-bounds}.
The argument in the proof of Theorem~\ref{thm:flow-separation} then
shows that, for every $\bm{x}_0\ne\bm{0}$ whose polar angle is not one
of these maximizers, the Clarke subgradient flow from $\bm{x}_0$ is
unique, avoids the origin, and converges to a strict local minimizer
$M\bm{u}(\beta)$, where $\beta$ is one of these minimizers. The excluded
initial points form four rays from the origin.}
\end{remark}

\begin{remark}
\label{rem:li-hu-flow}
\citet[Corollary~2.3]{LiHu2026} consider
\[
 H_{\bm{v}}(x,y):=
 s-\min\{x,y\}-2(-x-y)_+
 -\langle\bm{v},(x,y)\rangle+10\,(s-\tfrac{1}{10})_+,
\]
where $s:=\sqrt{x^2+y^2}$ and $\|\bm{v}\|<\tfrac{1}{200}$. On the set
$\{x+y>0,\ 0<s<\tfrac{1}{10}\}$, the terms involving  $(\cdot)_+$
vanish, so every Clarke subgradient of $H_{\bm{v}}$
has the form $\tfrac{(x,y)}{s}-(\lambda,1-\lambda)-\bm{v}$ with
$\lambda\in[0,1]$ and $\lambda x+(1-\lambda)y=\min\{x,y\}$.
Taking inner products of the velocity with $(1,1)$ and $\tfrac{(x,y)}{s}$,
and using $\min\{x,y\}\le \tfrac{x+y}{2}\leq\tfrac{s}{\sqrt2}$, every Clarke subgradient
trajectory satisfies
\[
 \frac{\mathrm{d}}{\mathrm{d}t}(x+y)=1+v_1+v_2-\frac{x+y}{s},
 \qquad
 \dot s=-\frac{H_{\bm{v}}(x,y)}{s}
 \le-\Bigl(1-\frac{1}{\sqrt2}-\|\bm{v}\|\Bigr)<0.
\]
The first identity prevents the trajectory from crossing
the line $x+y=0$ before reaching the origin, and the second
forces it to reach the origin in finite time.
Since the origin is Clarke
critical~\citep[Lemmas~B.3--B.4]{LiHu2026},
each such trajectory admits a stationary continuation
after arrival.
The initialization set
$\tfrac{1}{100}\mathbb{B}_{1/8}((1,0))$
used in their discrete counterexample lies in this region.
Thus their example admits Clarke subgradient flows that
reach the nonminimizing origin in finite time and remain
there, in contrast with Theorem~\ref{thm:flow-separation}.
\end{remark}
 
\section{Closing remarks}\label{sec:closing}
The example shows that, without subdifferential regularity,
semialgebraic geometry and bounded stochastic perturbations do not
suffice for generic convergence to local minimizers.
An attracting cycle in the rescaled iteration can instead drive the
iterates to a point that is not even limiting critical.
The gradient steps and the distance to this point vanish at the same
rate, so the discrete mechanism persists arbitrarily close to the limit.
Within positively homogeneous constructions, nonregularity is
necessary: if such a function is subdifferentially regular at the
origin and Clarke critical there, then the origin is Fr\'echet
critical, and homogeneity forces it to be a global minimizer.

The example also separates the discrete method from its
continuous-time counterpart: from the initial points that the
discrete method sends to the origin, the Clarke subgradient flow
converges to a strict local minimizer. This leaves open which
geometric conditions weaker than subdifferential regularity suffice
for generic convergence of stochastic subgradient methods to local
minimizers. By this separation, such conditions cannot be verified
through the behavior of the flow from these initial points alone;
they must also constrain the discrete dynamics near nonsmooth
critical points, for instance by excluding contracting cycles of the
rescaled iteration. 

\section*{Acknowledgments}
The author thanks Michael L. Overton for suggesting the
problem studied in this paper, for detailed discussions,
and for suggestions that improved the manuscript.
GPT-6 Astra assisted in finding {an explicit angular profile
satisfying the conditions in Section~\ref{sec:angular-profile}
derived from Proposition~\ref{prop:cycle-principle}}. AI tools also assisted with revision of the exposition. 
The author has independently verified all mathematical
arguments and {assumes responsibility for all content}.

\bibliographystyle{abbrvnat}
\bibliography{ref}
\end{document}